\documentclass[10pt,a4paper]{article}
\usepackage{amsfonts}
\usepackage{latexsym}
\usepackage{cite}
\usepackage{amsmath,amsfonts,latexsym,amssymb}
\usepackage[mathscr]{eucal}
\usepackage{cases}
\usepackage{amsthm}
\usepackage[bf,small]{caption2}
\usepackage{float,color}
\usepackage{graphicx}
\usepackage{amssymb}
\usepackage[all]{xy}

\newtheorem{theorem}{theorem}[section]

\newtheorem{conj}[theorem]{Conjecture}

\newtheorem{exmp}[theorem]{Example}
\newtheorem{lem}[theorem]{Lemma}
\newtheorem{nota}[theorem]{Notation}

\newtheorem{rmk}[theorem]{Remark}

\begin{document}

\title{\textbf{Computing twisted Alexander polynomials for Montesinos links}}
\author{\Large Haimiao Chen\footnote{\em{Email: chenhm@math.pku.edu.cn}} \\
Beijing Technology and Business University, Beijing, China \\
}
\date{}
\maketitle

\begin{abstract}
  In recent years, twisted Alexander polynomial has been playing an important role in low-dimensional topology.
  For Montesinos links, we develop an efficient method to compute the twisted Alexander polynomial associated to any linear representation.
  In particular, formulas for multi-variable Alexander polynomials of these links can be easily derived.

\end{abstract}

\section{Introduction}

Twisted Alexander polynomial (TAP for short) was first proposed by Lin \cite{Lin01} and Wada \cite{Wa94} independently.
%Given a link $L\subset S^3$ with $n$ components.
Given a link $L\subset S^3$. To each linear representation $\rho$ of $\pi(L):=\pi_1(S^3-L)$ is associated the TAP $\Delta_{L,\rho}$.
People have found TAP to be extremely useful in low-dimensional topology.

Friedl and Vidussi \cite{FV07} showed that TAP can detect the unknot and the Hopf link. They also constructed a pair
of knots with the same HOMFLY polynomial, Khovanov and knot Floer homology, but distinguished by TAP.

As a classical result of Fox and Milnor \cite{FM66}, if $K$ is slice, then the ordinary Alexander polynomial can be written as
$f(t)f(t^{-1})$ for some $f$. Generalized to TAP, more slice obstructions of such kind were obtained. Kirk and Livingston \cite{KL99}
showed that some knots (e.g. $8_{17}$) are not concordant to their inverses. Herald, Kirk and Livingston \cite{HKL10} showed that 16
of the knots with no more than 12 crossings are not slice; this could not be done before.

Fibredness can also be detected via TAP. Kitano and Morifuji \cite{KM05} proved that if $K$ is fibered, then $\Delta_{K,\rho}$
is monic of degree $4g(K)-2$ (where $g(K)$ is the genus of $K$) for any nonabelian $\rho:\pi(K)\to{\rm SL}(2,\mathbb{F})$ for any field $\mathbb{F}$.
Closely related is the following
\begin{conj} [Dunfield, Friedl and Jackson, 2012] \label{conj}
If $K$ is a hyperbolic knot and $\rho_{\rm holo}:\pi(K)\to{\rm SL}(2,\mathbb{C})$ is a lift of the discrete and faithful representation defining the hyperbolic metric, then $\deg\Delta_{K,\rho_{\rm holo}}=4g(K)-2$. Furthermore, $K$ is fibered if and only if $\Delta_{K,\rho_{\rm holo}}$ is monic.
\end{conj}
The following result of Friedl and Vidussi \cite{FV11} is remarkable: a knot $K$ is fibered if and only if there exist integers $a,b$ (which can be
expressed in terms of invariants of $K$) such that $\deg\Delta_{K,\rho}=a|G|+b$ for any representation $\rho$ factoring through a finite group $G$.

Another topic attracting much attention is the relation $K\ge K'$ defined by the existence of epimorphisms $\pi(K)\twoheadrightarrow\pi(K')$, which was known to be a partial order among prime knots.
As shown in \cite{KSW05}, if $\varphi:\pi(K)\twoheadrightarrow\pi(K')$ is an epimorphism, then $\Delta_{K,\rho\circ\varphi}/\Delta_{K',\rho}\in R[t^{\pm 1}]$ for any $\rho:\pi(K')\to{\rm GL}(d,R)$.

For more applications of TAP, see \cite{DFL14,FV08,Go16,HLN06,TY17} and the references therein.

TAP can be thought of as a non-abelian, yet manageable invariant.
TAPs of twist knots for nonabelian ${\rm SL}(2,\mathbb{C})$-representations were computed in \cite{Mo08} and interesting applications were found.
Hoste and Shanahan \cite{HS13} computed $\Delta_{K,\rho}$ for a class of 2-bridge knots $K$ and for $\rho$ factorizing through a dihedral
group. Morifuji and Tran \cite{MT14} computed TAPs associated to parabolic ${\rm SL}(2,\mathbb{C})$-representations of double twist knots
and verified Conjecture \ref{conj}. Tran \cite{Tr14} derived formulas for $\Delta_{K,\rho}$, when $K$ is a torus knot or a twist knot,
and $\rho$ is the adjoint representation associated to an ${\rm SL}(2,\mathbb{C})$-representation.
Besides these works, in the literature there are few computations for families of links or families of representations.
Indeed, unlike the ordinary Alexander polynomial, in general there is no ``skein relation" to use. For a link with $m$ crossings and a $d$-dimensional linear representation, to obtain TAP by definition, one needs to compute the determinant of a matrix of size $(m-1)d\times(m-1)d$, which is a tremendous job.

However, when restricted to Montesinos links $L$, the situation is different. We shall develop an efficient method for deriving a ``universal" formula for the TAP associated to any linear representation.
For a $d$-dimensional linear representation $\rho$, we may reduce the computation of $\Delta_{L,\rho}$ to those among $d\times d$-matrices.
This method will be useful for knot theorists to do computations. After all, Montesinos links, especially double twist links
and pretzel links, are usually taken as examples to support general results, sometimes a new conjecture, as done in \cite{As18,MT14,MT17}.
Specialized to the trivial representation, multi-variable Alexander polynomials are quickly computed, which will be beneficial to the realization problem of Alexander polynomials (see \cite{St10} Section 9).
Furthermore, till now the relation $\ge$ between Montesinos knots is not well understood, neither is ``1-domination" which is an enhance of $\ge$ (see \cite{BBRW15}); we expect TAP to promote studying these relations.

The content is organized as follows. Section 2 is a preliminary, introducing some basic notions and recalling TAP. In Section 3 we present the method, and in Section 4 we give several practical examples for illustration.

\medskip

{\bf Acknowledgement}  \\
This work is supported by NSFC-11401014, NSFC-11501014, NSFC-11771042.

\section{Preliminary}

\begin{nota}
\rm
For a (not necessarily commutative) ring $R$, let $R^\times$ denote the set of invertible elements.
For positive integers $d,d'$, let $\mathcal{M}^{d'}_{d}(R)$ denote the set of
$d\times d'$ matrices with entries in $R$. Let $\mathcal{M}_d(R)=\mathcal{M}^d_{d}(R)$.

Given $A\in\mathcal{M}^{d'}_{d}(R)$. Let $A_{\neg i}\in\mathcal{M}^{d'}_{d-1}(R)$ (resp. $A^{\neg j}\in\mathcal{M}^{d'-1}_{d}(R)$)
be the matrix obtained by deleting the $i$-th row (resp. $j$-th column) of $A$.
For a ring homomorphism $\psi:R\to R'$ and $A\in\mathcal{M}^{d'}_{d}(R)$, let $\psi(A)\in\mathcal{M}^{d'}_{d}(R')$ denote the matrix
obtained from $A$ by replacing each entry with its image under $\psi$.

For $k\in\mathbb{Z}$ and $a\in R^\times$, put
$$[k]_a=\begin{cases} 1+a+\cdots +a^{k-1}, & k>0, \\ 0, &k=0, \\ -a^{k}(1+a+\cdots+a^{|k|-1}), &k<0.\end{cases}$$

Given a group homomorphism $\sigma:G\to H$, let $\check{\sigma}:\mathbb{Z}G\to\mathbb{Z}H$ denote the canonical linear extension,
which is a ring homomorphism.

For two elements $g,h$ of a group, let $g.h=ghg^{-1}$.
\end{nota}

\subsection{Twisted Alexander polynomial}

Suppose $L$ is an oriented link with $n$ components.
Take a planar diagram for $L$. Numerating arbitrarily, let $x_1,\ldots,x_m$ be the directed arcs, and $\kappa_1,\ldots,\kappa_m$ the crossings.
Let $F_m$ be the free group on the letters $\tilde{x}_1,\ldots,\tilde{x}_m$, with $\tilde{x}_j$ corresponding to $x_j$.
With the ``over presentation" (see \cite{CF77} Chapter VI) for $\pi=\pi_1(L)$,
each positive crossing as in Figure \ref{fig:pm} (a) gives rise to a relator $\tilde{x}_j\tilde{x}_k\tilde{x}_j^{-1}\tilde{x}_\ell^{-1}$, and each negative crossing as in Figure \ref{fig:pm} (b) gives rise to a relator $\tilde{x}_j^{-1}\tilde{x}_k\tilde{x}_j\tilde{x}_\ell^{-1}$. By abuse of notation, we always denote the element of $\pi$ represented by a directed arc $x$ also by $x$. Let $\phi:F_m\twoheadrightarrow\pi$ be the quotient map sending $\tilde{x}_j$ to $x_j$. Let $\widetilde{M}\in\mathcal{M}_m(\mathbb{Z}F_m)$ be the matrix whose $(i,j)$-entry is $\partial\tilde{r}_i/\partial\tilde{x}_j$ (the Fox derivative),
and let $M=\check{\phi}(\widetilde{M})\in\mathcal{M}_{m}(\mathbb{Z}\pi)$.
Let
$$\alpha:\pi\twoheadrightarrow H_1(S^3-L)\cong\mathbb{Z}^n=\langle t_1,\ldots,t_n\rangle$$
denote the abelianzation, with $t_k$ the image of the $k$-th meridian.

\begin{figure}[h]
  \centering
  \includegraphics[width=0.5\textwidth]{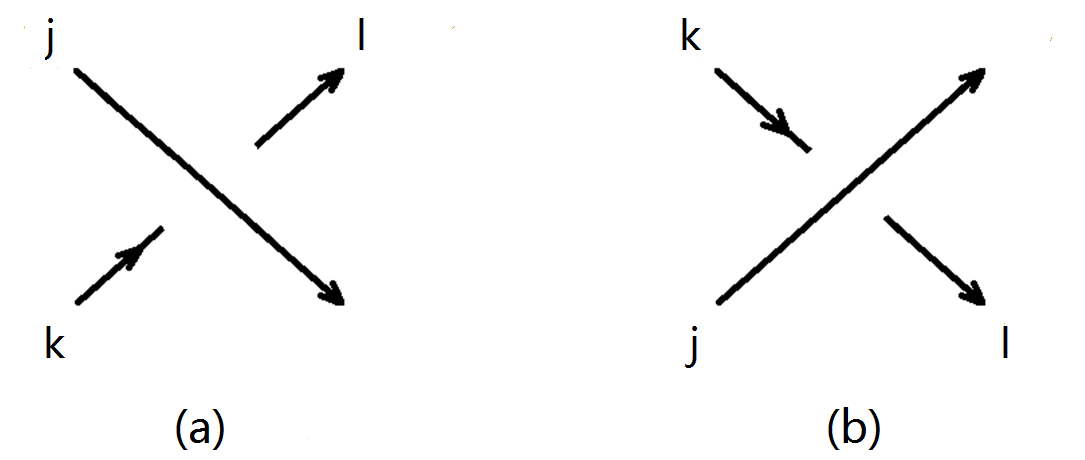}\\
  \caption{(a) a positive crossing $(x_j,x_k,x_{\ell})$; (b) a negative crossing $(x_j,x_k,x_{\ell})$} \label{fig:pm}
\end{figure}

Suppose $R$ is an integral domain. Let $R[\mathbf{t}^{\pm}]=R[t_1^{\pm 1},\ldots,t_n^{\pm 1}]$, and let $R(\mathbf{t}^{\pm})$ denote its fraction field. For $f_1,f_2\in R(\mathbf{t}^{\pm})$,
denote $f_1\doteq f_2$ if $f_2=\epsilon t_1^{e_1}\cdots t_n^{e_n}f_1$ for some $\epsilon\in R^\times$ and $e_1,\ldots,e_n\in\mathbb{Z}$;
clearly $\doteq$ is an equivalence relation.
Given a representation $\rho:\pi\to {\rm GL}(d,R)$, the obvious map
$$\pi\stackrel{\alpha\times\rho}\longrightarrow\mathbb{Z}^n\times {\rm GL}(d,R)\hookrightarrow R[\mathbf{t}^{\pm1}]\times\mathcal{M}_d(R)\to\mathcal{M}_d(R[\mathbf{t}^{\pm}])$$
can be extended by linearity to a ring homomorphism
\begin{align}
\mathbb{Z}\pi\to\mathcal{M}_d(R[\mathbf{t}^{\pm}]), \qquad c\mapsto\overline{c}, \label{eq:map}
\end{align}
Following \cite{Wa94}, define the {\it twisted Alexander polynomial} by
$$\Delta_{L,\rho}\doteq\frac{\det\left(\overline{M_{\neg i}^{\neg j}}\right)}{\det(1-\overline{x_j})}\in R(\mathbf{t}^{\pm}),$$
with the understanding that, up to equivalence, the right-hand-side of $\doteq$ is independent of all the choices.

Let $\mathbf{1}$ denote the (1-dimensional) trivial representation. It is known that
\begin{align}
\Delta_{L,\mathbf{1}}\doteq\begin{cases} \Delta_L/(t-1), & n=1, \\ \Delta_L, &n\ge 2, \end{cases}  \label{eq:ordinary}
\end{align}
where $\Delta_L$ is the ordinary Alexander polynomial if $n=1$, and is the multi-variable Alexander polynomial if $n\ge 2$.

\subsection{Montesinos tangle and Montesinos link} \label{sec:Montesinos}

\begin{figure} [h]
  \centering
  \includegraphics[width=0.27\textwidth]{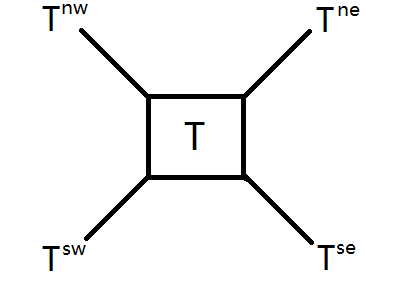}\\
  \caption{A tangle diagram with four ends} \label{fig:T}
\end{figure}

\begin{figure} [h]
  \centering
  \includegraphics[width=0.65\textwidth]{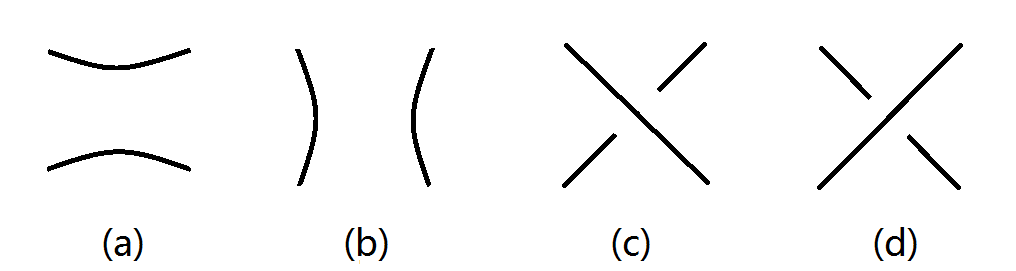}\\
  \caption{Four tangles in $\mathcal{T}_2^2$: (a) $[0]$, (b) $[\infty]$, (c) $[1]$, (d) $[-1]$}\label{fig:basic}
\end{figure}

\begin{figure}[h]
  \centering
  \includegraphics[width=0.55\textwidth]{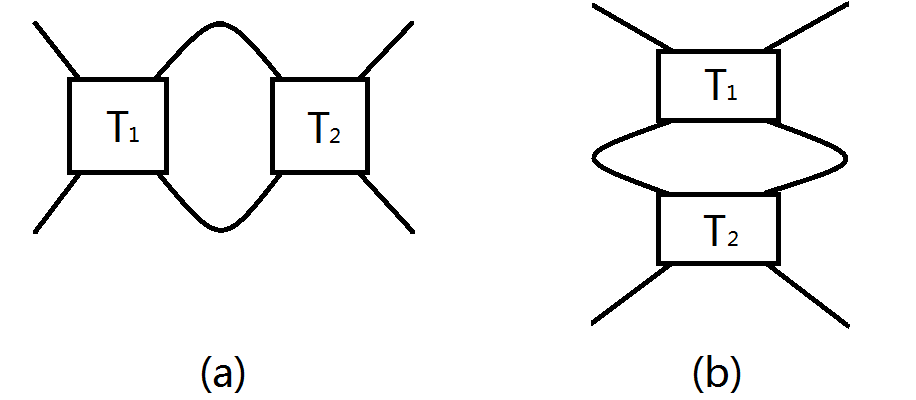}\\
  \caption{(a) $T_{1}+T_{2}$; (b) $T_{1}\ast T_{2}$} \label{fig:composition}
\end{figure}

Let $\mathcal{T}_2^2$ denote the set of tangles $T$ with four ends $T^{\rm nw}, T^{\rm ne}, T^{\rm sw}, T^{\rm se}$; see Figure \ref{fig:T}.
The simplest tangles in $\mathcal{T}_2^2$ are shown in Figure \ref{fig:basic}. Defined on $\mathcal{T}_2^2$ are horizontal composition $+$ and vertical composition $\ast$, as illustrated in Figure \ref{fig:composition}.

For $k\ne 0$, the horizontal composite of $|k|$ copies of $[1]$ (resp. $[-1]$) is denoted by $[k]$ if $k>0$ (resp. $k<0$), and the
vertical composite of $|k|$ copies of $[1]$ (resp. $[-1]$) is denoted by $[1/k]$ if $k>0$ (resp. $k<0$).
For $k_{1},\ldots,k_{r}\in\mathbb{Z}$, the {\it rational tangle} $[[k_{1}],\ldots,[k_{r}]]$ by definition is
$$\begin{cases}
[k_{1}]\ast [1/k_{2}]+\cdots+[k_{r}], &\text{if\ \ } 2\nmid r, \\
[1/k_{1}]+[k_{2}]\ast\cdots +[k_{r}], &\text{if\ \ } 2\mid r.
\end{cases}$$
Its {\it fraction} is given by the continued fraction $[[k_{1},\ldots,k_{r}]]\in\mathbb{Q}$, which is defined inductively as
$$[[k_{1}]]=k_{1}; \qquad
[[k_{1},\ldots,k_{j}]]=k_{j}+1/[[k_{1},\ldots,k_{j-1}]], \quad j\ge 2.$$
Denote $[[k_{1}],\ldots,[k_{r}]]$ as $[p/q]$ when its fraction equals $p/q$.

\begin{figure}[h]
  \centering
  \includegraphics[width=0.5\textwidth]{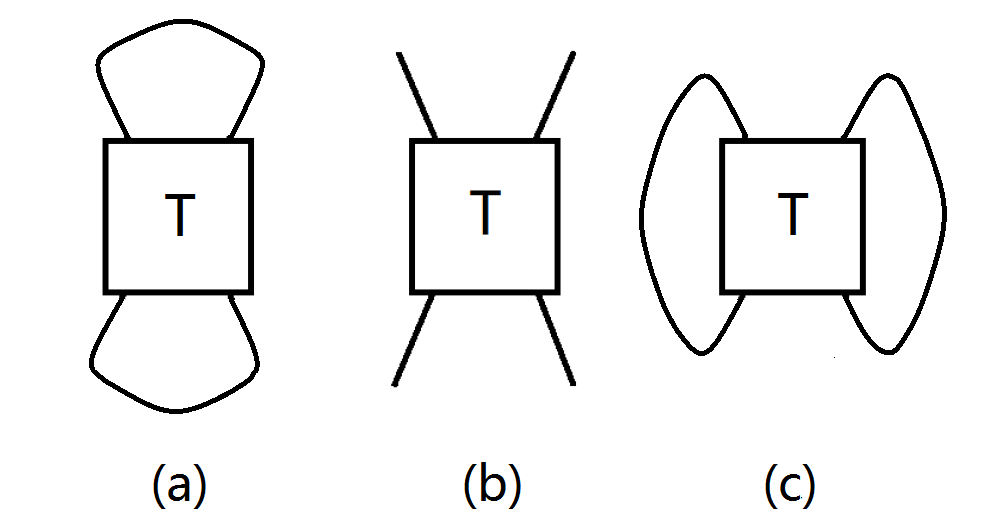}\\
  \caption{(a) the numerator $N(T)$; (b) a tangle $T$; (c) the denominator $D(T)$} \label{fig:D-N}
\end{figure}

A {\it Montesinos tangle} is one of the form $[p_1/q_1]\ast\cdots\ast[p_r/q_r]$.

Each $T\in\mathcal{T}_2^2$ gives rise to two links: the {\it numerator} $N(T)$ and the {\it denominator} $D(T)$; see Figure \ref{fig:D-N}. For $p/q\in\mathbb{Q}$, the link $N([p/q])$ is called a {\it rational link}
or {\it 2-bridge link}. In particular, $N([[k_1],[k_2]])$ is called a {\it double twist link}, and $N([[k_1],[k_2],[k_3]])$ is called a {\it triple twist link}.
A link of the form $D([p_1/q_1]\ast\cdots\ast[p_s/q_s])$ is called a {\it Montesinos link} and denoted by $M(p_1/q_1,\ldots,p_s/q_s)$;
in particular, $M(p_1,\ldots,p_s)$ is a {\it pretzel link}.

\section{The method}

The main results are (\ref{eq:TAP-N(T)}) and (\ref{eq:main}), which are formulas of TAPs for rational links and general Montesinos links, respectively.

\begin{nota}
\rm For a crossing $\kappa_i=(x_j,x_k,x_{\ell})$ as given in Figure \ref{fig:pm}, denote $j, k, \ell$ by $\overline{i}, \underline{i}, i'$, respectively.
\end{nota}

If $\kappa_i$ is positive so that $\tilde{r}_i=\tilde{x}_{\overline{i}}\tilde{x}_{\underline{i}}\tilde{x}_{\overline{i}}^{-1}\tilde{x}_{i'}^{-1}$, then
$$\frac{\partial \tilde{r}_i}{\partial\tilde{x}_{\overline{i}}}=1-\tilde{x}_{\overline{i}}\tilde{x}_{\underline{i}}\tilde{x}_{\overline{i}}^{-1}, \qquad
\frac{\partial\tilde{r}_i}{\partial\tilde{x}_{\underline{i}}}=\tilde{x}_{\overline{i}}, \qquad
\frac{\partial \tilde{r}_i}{\partial\tilde{x}_{i'}}=-\tilde{x}_{\overline{i}}\tilde{x}_{\underline{i}}\tilde{x}_{\overline{i}}^{-1}\tilde{x}_{i'}^{-1},$$
hence
$$\check{\phi}\Big(\frac{\partial \tilde{r}_i}{\partial\tilde{x}_{\overline{i}}}\Big)=1-x_{i'}, \qquad
\check{\phi}\Big(\frac{\partial\tilde{r}_i}{\partial\tilde{x}_{\underline{i}}}\Big)=x_{\overline{i}}, \qquad
\check{\phi}\Big(\frac{\partial \tilde{r}_i}{\partial\tilde{x}_{i'}}\Big)=-1,$$
so that, written in matrices,
$$\check{\phi}\left[\frac{\partial\tilde{r}_i}{\partial\tilde{x}_{\overline{i}}},\frac{\partial\tilde{r}_i}{\partial\tilde{x}_{\underline{i}}},
\frac{\partial\tilde{r}_i}{\partial\tilde{x}_{i'}}\right]D[1-x_{\overline{i}},1-x_{\underline{i}},1-x_{i'}]=(1-x_{i'})[1-x_{\overline{i}},x_{\overline{i}},-1],$$
where $D[a_1,a_2,a_3]$ denotes the $3\times 3$ diagonal matrix with diagonal entries $a_i$.

If $\kappa_i$ is negative so that $\tilde{r}_i=\tilde{x}_{\overline{i}}^{-1}\tilde{x}_{\underline{i}}\tilde{x}_{\overline{i}}\tilde{x}_{i'}^{-1}$,
then
$$\frac{\partial\tilde{r}_i}{\partial\tilde{x}_{\overline{i}}}=-\tilde{x}_{\overline{i}}^{-1}+\tilde{x}_{\overline{i}}^{-1}\tilde{x}_{\underline{i}}, \qquad
\frac{\partial\tilde{r}_i}{\partial\tilde{x}_{\underline{i}}}=\tilde{x}_{\overline{i}}^{-1},\qquad
\frac{\partial \tilde{r}_i}{\partial\tilde{x}_{i'}}=-\tilde{x}_{\overline{i}}^{-1}\tilde{x}_{\underline{i}}\tilde{x}_{\overline{i}}\tilde{x}_{i'}^{-1},$$
hence
$$\check{\phi}\Big(\frac{\partial\tilde{r}_i}{\partial\tilde{x}_{\overline{i}}}\Big)=-x_{\overline{i}}^{-1}+x_{i'}x_{\overline{i}}^{-1}, \qquad
\check{\phi}\Big(\frac{\partial\tilde{r}_i}{\partial\tilde{x}_{\underline{i}}}\Big)=x_{\overline{i}}^{-1},\qquad
\check{\phi}\Big(\frac{\partial \tilde{r}_i}{\partial\tilde{x}_{i'}}\Big)=-1,$$
so that
$$\check{\phi}\left[\frac{\partial\tilde{r}_i}{\partial\tilde{x}_{\overline{i}}},\frac{\partial\tilde{r}_i}{\partial\tilde{x}_{\underline{i}}},
\frac{\partial\tilde{r}_i}{\partial\tilde{x}_{i'}}\right]D[1-x_{\overline{i}},1-x_{\underline{i}},1-x_{i'}]=(1-x_{i'})[1-x_{\overline{i}}^{-1},x_{\overline{i}}^{-1},-1].$$

Consequently,
\begin{align}
MD=D'Q,  \label{eq:transform}
\end{align}
where
\begin{enumerate}
  \item[\rm(i)] $D\in\mathcal{M}_m(\mathbb{Z}\pi)$ is diagonal, whose $j$-th diagonal entry is $1-x_{j}$;
  \item[\rm(ii)] $D'\in\mathcal{M}_m(\mathbb{Z}\pi)$ is diagonal, whose $i$-th diagonal entry is $1-x_{i'}$;
  \item[\rm(iii)] $Q_{i,\overline{i}}=1-x_{\overline{i}}^{\epsilon}$, $Q_{i,\underline{i}}=x_{\overline{i}}^{\epsilon}$, $Q_{i,i'}=-1$, where $\epsilon=1$ (resp. $\epsilon=-1$) if $\kappa_i$ is positive (resp. negative), and $Q_{i,j}=0$ for $j\notin\{\overline{i},\underline{i},i'\}$.
\end{enumerate}

As a consequence of (\ref{eq:transform}), for any $i,j$, one has $M_{\neg i}^{\neg j}D_{\neg j}^{\neg j}=(D')^{\neg i}_{\neg i}Q_{\neg i}^{\neg j}$, so
\begin{align}
\Delta_{L,\rho}
\doteq\frac{\det\left(\overline{M_{\neg i}^{\neg j}}\right)}{\det(1-\overline{x_{j}})}
\doteq\frac{\det\left(\overline{Q_{\neg i}^{\neg j}}\right)}{\det(1-\overline{x_{i'}})}
\doteq\frac{\det\left(\overline{Q_{\neg i}^{\neg j}}\right)}{\det(1-\overline{x_{\underline{i}}})},  \label{eq:Q}
\end{align}
the last equality following from that $x_{i'}$ is conjugate to $x_{\underline{i}}$.

We shall do elementary transformations to convert $Q$ into another matrix which is highly simplified.
To speak conveniently, introduce for each direct arc $z$ a formal variable $\xi_z$.
Each row of $Q$ corresponds to an equation $\sum_{j=1}^mQ_{i,j}\xi_{x_j}=0,$ and doing row-transformations is equivalent to re-writing equations of such kind.
Note that $\xi_z$ depends only on the underlying arc of $z$, so $\xi_{z'}$ also makes sense for an un-directed arc $z'$.

\begin{figure}[h]
  \centering
  \includegraphics[width=0.7\textwidth]{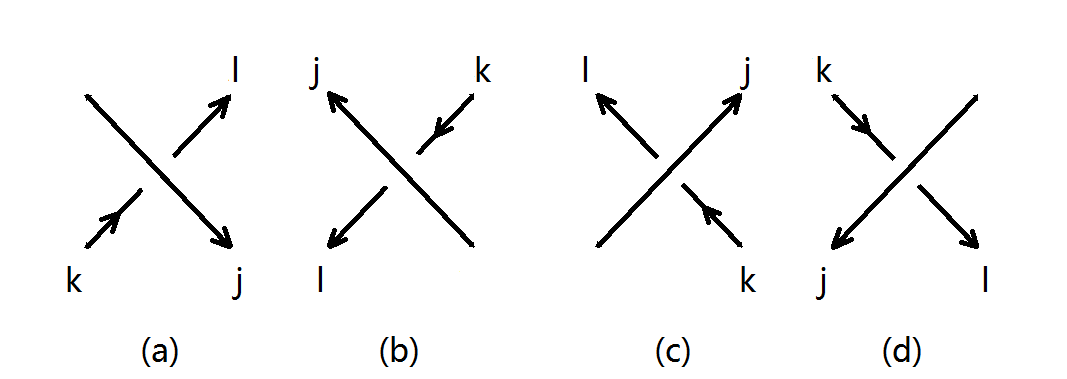}\\
  \caption{Four possibilities of a positive crossing}
  \label{fig:positive}
\end{figure}
\begin{figure}[h]
  \centering
  \includegraphics[width=0.7\textwidth]{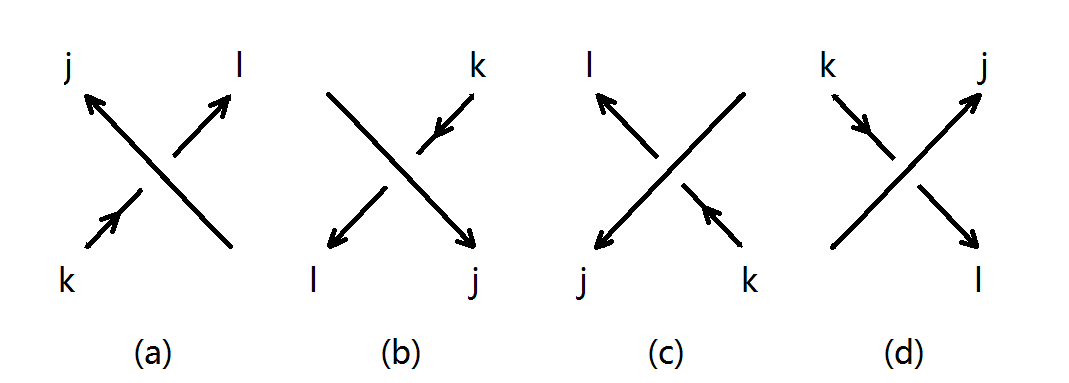}\\
  \caption{Four possibilities of a negative crossing}\label{fig:negative}
\end{figure}

For each crossing, let $\xi_{\rm nw}$ denote the variable associated to the northwest arc, and so on.
Adopt the following abbreviations:
$$\xi_{\mathbf{w}}=\left[\begin{array}{cc} \xi_{\rm nw} \\ \xi_{\rm sw} \end{array}\right], \quad
\xi_{\mathbf{e}}=\left[\begin{array}{cc} \xi_{\rm ne} \\ \xi_{\rm se} \end{array}\right], \quad
\xi_{\mathbf{n}}=\left[\begin{array}{cc} \xi_{\rm nw} \\ \xi_{\rm ne} \end{array}\right], \quad
\xi_{\mathbf{s}}=\left[\begin{array}{cc} \xi_{\rm sw} \\ \xi_{\rm se} \end{array}\right].$$

For a positive crossing shown in Figure \ref{fig:positive}, $\xi_\ell=(1-x_j)\xi_j+x_j\xi_k$, so in case (a), (b), (c), (d), respectively,
\begin{align}
&\xi_{\mathbf{e}}=\left[\begin{array}{cc} 1-x_j & x_j \\ 1 & 0 \end{array}\right]\xi_{\mathbf{w}},
&&\xi_{\mathbf{s}}=\left[\begin{array}{cc} 1-x_j^{-1} & x_j^{-1} \\ 1 & 0 \end{array}\right]\xi_{\mathbf{n}}, \label{eq:positive1} \\
&\xi_{\mathbf{e}}=\left[\begin{array}{cc} 1-x_j^{-1} & x_j^{-1} \\ 1 & 0 \end{array}\right]\xi_{\mathbf{w}},
&&\xi_{\mathbf{s}}=\left[\begin{array}{cc} 1-x_j & x_j \\ 1 & 0 \end{array}\right]\xi_{\mathbf{n}}, \label{eq:positive2}  \\
&\xi_{\mathbf{e}}=\left[\begin{array}{cc} 0 & 1 \\ x_j^{-1} & 1-x_j^{-1} \end{array}\right]\xi_{\mathbf{w}},
&&\xi_{\mathbf{s}}=\left[\begin{array}{cc} 0 & 1 \\ x_j^{-1} & 1-x_j^{-1} \end{array}\right]\xi_{\mathbf{n}}, \label{eq:positive3}  \\
&\xi_{\mathbf{e}}=\left[\begin{array}{cc} 0 & 1 \\ x_j & 1-x_j \end{array}\right]\xi_{\mathbf{w}},
&&\xi_{\mathbf{s}}=\left[\begin{array}{cc} 0 & 1 \\ x_j & 1-x_j \end{array}\right]\xi_{\mathbf{n}}. \label{eq:positive4}
\end{align}
For a negative crossing shown in Figure \ref{fig:negative}, $\xi_\ell=(1-x_j^{-1})\xi_j+x_j^{-1}\xi_k$, so in case (a), (b), (c), (d), respectively,
\begin{align}
&\xi_{\mathbf{e}}=\left[\begin{array}{cc} 1-x_j^{-1} & x_j^{-1} \\ 1 & 0 \end{array}\right]\xi_{\mathbf{w}},
&&\xi_{\mathbf{s}}=\left[\begin{array}{cc} 1-x_j & x_j \\ 1 & 0 \end{array}\right]\xi_{\mathbf{n}},  \label{eq:negative1}  \\
&\xi_{\mathbf{e}}=\left[\begin{array}{cc} 1-x_j & x_j \\ 1 & 0 \end{array}\right]\xi_{\mathbf{w}},
&&\xi_{\mathbf{s}}=\left[\begin{array}{cc} 1-x_j^{-1} & x_j^{-1} \\ 1 & 0 \end{array}\right]\xi_{\mathbf{n}}, \label{eq:negative2} \\
&\xi_{\mathbf{e}}=\left[\begin{array}{cc} 0 & 1 \\ x_j & 1-x_j \end{array}\right]\xi_{\mathbf{w}},
&&\xi_{\mathbf{s}}=\left[\begin{array}{cc} 0 & 1 \\ x_j & 1-x_j \end{array}\right]\xi_{\mathbf{n}}, \label{eq:negative3} \\
&\xi_{\mathbf{e}}=\left[\begin{array}{cc} 0 & 1 \\ x_j^{-1} & 1-x_j^{-1} \end{array}\right]\xi_{\mathbf{w}},
&&\xi_{\mathbf{s}}=\left[\begin{array}{cc} 0 & 1 \\ x_j^{-1} & 1-x_j^{-1} \end{array}\right]\xi_{\mathbf{n}}. \label{eq:negative4}
\end{align}

\begin{nota}
\rm For a rational tangle $T=[[k_1],\ldots,[k_r]]$, introduce the ``convenient labeling" for some of its directed arcs; see Figure \ref{fig:rational}.
\end{nota}

\begin{figure}[h]
  \centering
  \includegraphics[width=0.85\textwidth]{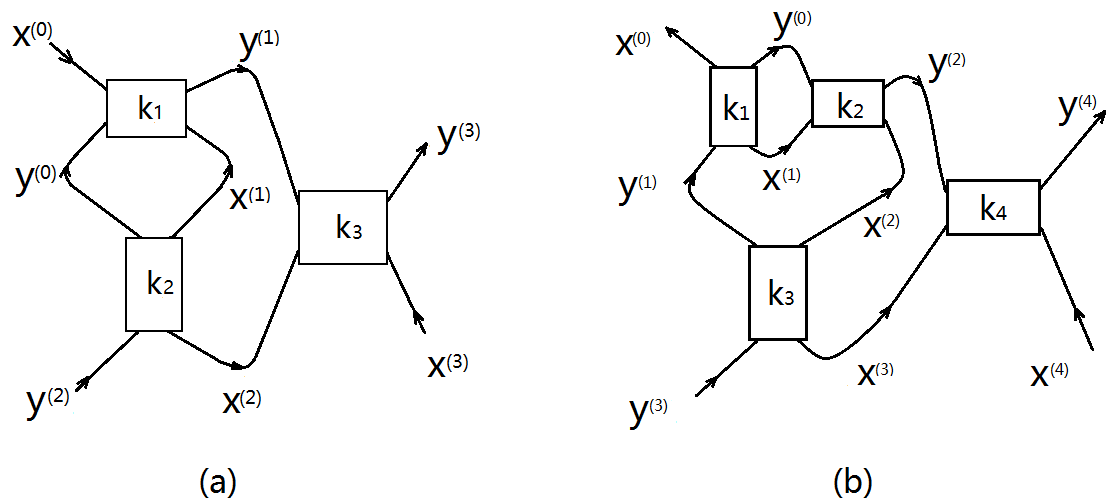}\\
  \caption{The ``convenient labeling" of a rational tangle: (a) $r$ is odd; (b) $r$ is even}\label{fig:rational}
\end{figure}

For $k\in\mathbb{Z}$ and $a,b\in\pi$, put
\begin{align*}
u_{k}(a,b)&=\begin{cases} [h]_{ab}a(1-b), &k=2h, \\ [h+1]_{ab}a-[h]_{ab}ab, &k=2h+1, \end{cases}  \\
G_{k}(a,b)&=\left[\begin{array}{cc} 1-u_{k}(a,b) & u_{k}(a,b) \\ 1-u_{k-1}(a,b) & u_{k-1}(a,b) \\ \end{array}\right].
\end{align*}
More explicitly,
\begin{align*}
G_{2h}(a,b)&=\left[\begin{array}{cc} 1-[h]_{ab}(a-ab) & [h]_{ab}(a-ab) \\ -[h]_{ab}(a-1) & 1+[h]_{ab}(a-1) \\ \end{array}\right], \\
G_{2h+1}(a,b)&=\left[\begin{array}{cc} -[h+1]_{ab}(a-1) & 1+[h+1]_{ab}(a-1) \\ 1-[h]_{ab}(a-ab) & [h]_{ab}(a-ab) \\ \end{array}\right].
\end{align*}

Using (\ref{eq:positive1})--(\ref{eq:negative4}) iteratively, we obtain
$$\left[\begin{array}{cc} \xi_{y^{(j)}} \\ \xi_{x^{(j)}}\end{array}\right]=G_{k_j}(y^{(j-2)},x^{(j-1)})\left[\begin{array}{cc} \xi_{y^{(j-2)}} \\ \xi_{x^{(j-1)}}\end{array}\right],$$

The following is tautological:
\begin{lem} \label{lem:algorithm}
Suppose $T=[[k_1],\ldots,[k_r]]$. Let $\xi'=\xi_{x^{(0)}}, \xi''=\xi_{y^{(0)}}$. Then
$$\xi_{T^{\rm ne}}=(1-b^{\rm ne})\xi'+b^{\rm ne}\xi'', \quad \xi_{T^{\rm sw}}=(1-b^{\rm sw})\xi'+b^{\rm sw}\xi'', \quad \xi_{T^{\rm se}}=(1-b^{\rm se})\xi'+b^{\rm se}\xi'',$$
with $b^{\rm ne}=\eta_{y^{(r)}}, b^{\rm sw}=\eta_{y^{(r-1)}}, b^{\rm se}=\eta_{x^{(r)}}$, where $\eta_{x^{(j)}}, \eta_{y^{(j)}}$ are recursively given by
$\eta_{y^{(0)}}=1$, $\eta_{y^{(1)}}=u_{k_1}(x^{(0)},y^{(0)})$, $\eta_{x^{(1)}}=u_{k_1-1}(x^{(0)},y^{(0)})$, and
$$\left[\begin{array}{cc} \eta_{y^{(j)}} \\ \eta_{x^{(j)}}\end{array}\right]=G_{k_j}(y^{(j-2)},x^{(j-1)})\left[\begin{array}{cc} \eta_{y^{(j-2)}} \\ \eta_{x^{(j-1)}}\end{array}\right], \qquad j\ge 2.$$
\end{lem}

\bigskip

Suppose $L=N(T)$, with $T=[[k_1],\ldots,[k_r]]$. Let $m=|k_1|+\cdots+|k_r|$. Define $Q_T\in\mathcal{M}_m^{m+2}(\mathbb{Z}\pi)$ by setting each row to fit the condition (iii) following (\ref{eq:transform}).
Then each $\xi_{x_j}$ can be written as $(1-b_j)\xi'+b_j\xi''$ for some $b_j\in\mathbb{Z}\pi$. As effects on matrices, re-numbering columns beforehand if necessary, $Q_T$ can be converted via a series of row transformations into
\begin{align*}
UQ_T=
\left[\begin{array}{ccccc}
1-b_1 & b_1 & -1 & \ & \  \\
\vdots & \vdots & \ & \ddots & \ \\
1-b_{m} & b_{m} & \ & \ & -1 \\
\end{array}\right],
\end{align*}
where $b_{i^{\rm ne}}=b^{\rm ne}, b_{i^{\rm sw}}=b^{\rm sw}, b_{i^{\rm se}}=b^{\rm se}$ for some $i^{\rm ne}, i^{\rm sw}, i^{\rm se}\in\{1,\ldots,m\}$.

More precisely, as seen from (\ref{eq:positive1})--(\ref{eq:negative4}), each of the row transformations used is or can be decomposed into {\it special transformations}, by which we mean those belong to the following:
\begin{enumerate}
  \item[\rm1)] multiplying the $i$-th row by $-x_{i'}^{\pm1}$, for some $i$;
  \item[\rm2)] adding $(1-x_{\overline{i}}^{\pm1})$ times the $i$-th row to or subtracting $(1-x_{\overline{i}}^{\pm1})$ times the $i$-th row from another row, for some $i$.
\end{enumerate}

\begin{lem} \label{lem:special}
Suppose $f\in R[\mathbf{t}^{\pm1}]$, and suppose $P\in{\rm GL}(m,\mathbb{Z}\pi)$ satisfies
$$\det\Big(\overline{P_{\neg i}^{\neg j}}\Big)\doteq\det(1-\overline{x_{i'}})f$$
for all $i,j$, then so does $VP$, for any elementary matrix $V$ representing a special row transformation.
\end{lem}
\begin{proof}
It suffices to show that $\det\big(\overline{(VP)_{\neg i}^{\neg j}}\big)$ can be divided by $\det(1-\overline{x_{i'}})f$, which can be verified directly.
\end{proof}

Note that $UQ$ can be obtained from $UQ_T$ by deleting the columns corresponding to $T^{\rm ne}$, $T^{\rm se}$, and modifying the $i^{\rm ne}$-th and $i^{\rm se}$-th rows respectively according to
$\xi_{T^{\rm ne}}=\xi'$ and $\xi_{T^{\rm se}}=\xi_{T^{\rm sw}}$. Precisely, the $i^{\rm ne}$-th row of $UQ$ is $(-b^{\rm ne},b^{\rm ne},0,\ldots,0)$, and the $i^{\rm se}$-th row
is $(b^{\rm sw}-b^{\rm se},b^{\rm se}-b^{\rm sw},0,\ldots,0)$.
This shows
\begin{align}
\det\big(\overline{(UQ)_{\neg i^{\rm se}}^{\neg 1}}\big)\doteq\det\big(\overline{b^{\rm ne}}\big).  \label{eq:det-rational-1}
\end{align}
We already know (retaining (\ref{eq:Q})) that $\det\big(\overline{Q_{\neg i}^{\neg j}}\big)\doteq\det(1-\overline{x_{i'}})\Delta_{L,\rho}$ for all $i,j$, hence by Lemma \ref{lem:special},
\begin{align}
\det\big(\overline{Q_{\neg i^{\rm se}}^{\neg 1}}\big)\doteq\det\big(\overline{(UQ)_{\neg i^{\rm se}}^{\neg 1}}\big)\doteq\det\big(\overline{b^{\rm ne}}\big).  \label{eq:det-rational-2}
\end{align}
Thus by (\ref{eq:Q}),
\begin{align}
\Delta_{N(T),\rho}\doteq\frac{\det\big(\overline{b^{\rm ne}}\big)}{\det(1-\overline{x^{\rm se}})},  \label{eq:TAP-N(T)}
\end{align}
where $x^{\rm se}$ is the directed arc whose underlying arc is $T^{\rm se}$.

\medskip

Now suppose $L=D(T)$, with $T=T_1\ast\cdots\ast T_s$, ($s\ge 2$), each $T_k$ being a rational tangle.
Let $m_k$ be the number of crossings in $T_k$, and let $m=m_1+\cdots+m_s$.

For each $k$, similarly as above, set up $Q_{T_k}$ and $b_k^{\rm ne}$, etc. Apply (\ref{eq:map}) to everything. To conveniently talk about row-transformations on $\overline{Q_{T_k}}$,
we introduce formal variables $\overline{\xi}_{x_j}, \overline{\xi}'_{k},\overline{\xi}''_{k}$, $\overline{\xi}_{T_k^{\rm ne}}$ and so on; abbreviate $\overline{\xi}_{T_k^{\rm ne}}$ to $\overline{\xi}_k^{\rm ne}$, etc.

Suppose $\overline{b_1^{\rm ne}},\ldots,\overline{b_s^{\rm ne}}$ are all invertible, i.e., $\det(\overline{b_k^{\rm ne}})\not\equiv0$ as elements of $R[\mathbf{t}^{\pm1}]$.
Re-write $\overline{\xi}_k^{\rm ne}=(1-\overline{b_k^{\rm ne}})\overline{\xi}'_k+\overline{b_k^{\rm ne}}\overline{\xi}''_{k}$ as
\begin{align}
\overline{\xi}''_{k}=(1-\overline{b_k^{\rm ne}}^{-1})\overline{\xi}'_k+\overline{b_k^{\rm ne}}^{-1}\overline{\xi}_k^{\rm ne}.  \label{eq:linear-combin}
\end{align}
With $\overline{\xi}''_{k}$ substituted, $\overline{\xi}_{x_j}$ for each $x_j$ in $T_k$ can be expressed as a linear combination of $\overline{\xi}'_k$ and $\overline{\xi}_k^{\rm ne}$;
in particular,
$$\left[\begin{array}{cc} \overline{\xi}_k^{\rm sw} \\ \overline{\xi}_k^{\rm se} \end{array}\right]=
\left[\begin{array}{cc} 1-\overline{b_k^{\rm sw}}\overline{b_k^{\rm ne}}^{-1} & \overline{b_k^{\rm sw}}\overline{b_k^{\rm ne}}^{-1} \\
1-\overline{b_k^{\rm se}}\overline{b_k^{\rm ne}}^{-1} & \overline{b_k^{\rm se}}\overline{b_k^{\rm ne}}^{-1}\end{array}\right]
\left[\begin{array}{cc} \overline{\xi}'_k \\ \overline{\xi}_k^{\rm ne} \end{array}\right].$$
The effect on $\overline{Q_{T_k}}$ is to multiply some $\overline{U}_k\in{\rm GL}(m_k,\mathcal{M}(d,R(\mathbf{t}^{\pm})))$ on the left. Note that $\det(\overline{U}_k)\doteq\det(\overline{b_k^{\rm ne}})^{-1}$, as we have multiplied $\overline{b_k^{\rm ne}}^{-1}$ to derive (\ref{eq:linear-combin}).

Since $\overline{\xi}'_k=\overline{\xi}_{k-1}^{\rm sw}$ and $\overline{\xi}_k^{\rm ne}=\overline{\xi}_{k-1}^{\rm se}$ for $k=2,\ldots,s$, by further substituting, we can re-write each $\overline{\xi}_{x_j}$ as a linear combination of $\overline{\xi}'_1$ and $\overline{\xi}_1^{\rm ne}$. What has been shown is that, re-numbering columns beforehand if necessary, we may take $\overline{U}\in{\rm GL}(m,\mathcal{M}(d,R(\mathbf{t}^{\pm})))$ such that $$\det(\overline{U})\doteq\prod_{k=1}^s\det(\overline{b_k^{\rm ne}})^{-1}, \qquad
\overline{U}\overline{Q}=\left[\begin{array}{ccc} \star & -I(m-2) & \ \\ B & \ & -I(2) \end{array}\right].$$
The last two rows of $\overline{U}\overline{Q}$ respectively express $\overline{\xi}_s^{\rm sw}$ and $\overline{\xi}_s^{\rm se}$ as a linear combination of $\overline{\xi}'_1$ and $\overline{\xi}_1^{\rm ne}$, i.e.,
$\left[\begin{array}{cc} \overline{\xi}_s^{\rm sw} \\ \overline{\xi}_s^{\rm se} \end{array}\right]=B\left[\begin{array}{cc} \overline{\xi}'_1 \\ \overline{\xi}_1^{\rm ne} \end{array}\right]$,
with
$$B=\left[\begin{array}{cc} 1-\overline{b_s^{\rm sw}}\overline{b_s^{\rm ne}}^{-1} & \overline{b_s^{\rm sw}}\overline{b_s^{\rm ne}}^{-1} \\
1-\overline{b_s^{\rm se}}\overline{b_s^{\rm ne}}^{-1} & \overline{b_s^{\rm se}}\overline{b_s^{\rm ne}}^{-1}\end{array}\right]\cdots
\left[\begin{array}{cc} 1-\overline{b_1^{\rm sw}}\overline{b_1^{\rm ne}}^{-1} & \overline{b_1^{\rm sw}}\overline{b_1^{\rm ne}}^{-1} \\
1-\overline{b_1^{\rm se}}\overline{b_1^{\rm ne}}^{-1} & \overline{b_1^{\rm se}}\overline{b_1^{\rm ne}}^{-1}\end{array}\right]=\left[\begin{array}{cc} 1-\overline{c} & \overline{c} \\ \star & \star \end{array}\right],$$
where $\overline{c}=\overline{p}_1+\overline{p}_2\overline{q}_1+\cdots+\overline{p}_s\overline{q}_{s-1}\cdots\overline{q}_1$, and
$$\overline{p}_k=\overline{b_k^{\rm sw}}\overline{b_k^{\rm ne}}^{-1}, \qquad \overline{q}_k=(\overline{b_k^{\rm se}}-\overline{b_k^{\rm sw}})\overline{b_k^{\rm ne}}^{-1}.$$
Equating $\overline{\xi}_s^{\rm sw}$, $\overline{\xi}_s^{\rm se}$ with $\overline{\xi}'_1$, $\overline{\xi}_1^{\rm ne}$, respectively, similarly as (\ref{eq:det-rational-1}), we obtain
$$\det(\overline{(UQ)_{\neg m}^{\neg 1}})\doteq\det(\overline{c})\cdot\prod\limits_{k=1}^s\det(\overline{b^{\rm ne}_{k}}).$$
Note that in above, each of row transformations other than multiplying via $\overline{b_k^{\rm ne}}^{-1}$ can be decomposed into special ones. Similarly as (\ref{eq:det-rational-2}), by Lemma \ref{lem:special},
\begin{align}
\det(\overline{Q_{\neg m}^{\neg 1}})\doteq\det(\overline{(UQ)_{\neg m}^{\neg 1}})\doteq\det(\overline{c})\cdot\prod\limits_{k=1}^s\det(\overline{b^{\rm ne}_{k}}). \label{eq:main-det}
\end{align}

Therefore, letting $x_s^{\rm se}$ denote the directed arc whose underlying arc is $T_s^{\rm se}$,
\begin{align}
\Delta_{L,\rho}\doteq\frac{\prod_{k=1}^s\det(\overline{b^{\rm ne}_{k}})}{\det(1-\overline{x_s^{\rm se}})}\det(\overline{p}_1+\overline{p}_2\overline{q}_1+\cdots+\overline{p}_s\overline{q}_{s-1}\cdots\overline{q}_1).  \label{eq:main}
\end{align}
In particular, when $s=3$,
\begin{align}
\Delta_{L,\rho}\doteq\frac{\det(\overline{b_2^{\rm ne}})\det(\overline{b_3^{\rm ne}})}{\det(1-\overline{x_3^{\rm se}})}
\det\Big(\overline{b_1^{\rm sw}}+\big(\overline{b_2^{\rm sw}}+\overline{b_3^{\rm sw}}\overline{b_3^{\rm ne}}^{-1}(\overline{b_2^{\rm se}}-\overline{b_2^{\rm sw}})\big)
\overline{b_2^{\rm ne}}^{-1}(\overline{b_1^{\rm se}}-\overline{b_1^{\rm sw}})\Big).  \label{eq:3-strand}
\end{align}

To deal with the ``degenerate" case when $\det(\overline{b^{\rm ne}_{k}})\equiv 0$ for some $k$, we introduce auxiliary variables $\varepsilon_k$ that commute with everything, replace $\overline{b^{\rm ne}_{k}}$ by $\overline{b^{\rm ne}_{k}}+\varepsilon_kI$, and run the process from (\ref{eq:linear-combin}) to (\ref{eq:main-det}). Then we arrive at
\begin{align}
\det(\overline{c}')\cdot\prod\limits_{k=1}^s\det(\overline{b^{\rm ne}_{k}}+\varepsilon_kI), \label{eq:degenerate}
\end{align}
(where $\overline{c}'$ is obtained from $\overline{c}$ via replacing $\overline{b^{\rm ne}_{k}}$ by $\overline{b^{\rm ne}_{k}}+\varepsilon_kI$), which can be expanded as a polynomial in $\varepsilon_1,\ldots,\varepsilon_n$, since what is being computed is the determinant of some matrix whose entries are all such polynomials. Setting $\varepsilon_1=\cdots=\varepsilon_n=0$ in (\ref{eq:degenerate}) and dividing by $\det(1-\overline{x_s^{\rm se}})$, we will obtain an expression for $\Delta_{L,\rho}$.

\begin{rmk}
\rm We believe that each $\overline{b^{\rm ne}_{k}}$ is actually always invertible, but temporarily cannot prove this.
\end{rmk}

\begin{rmk}
\rm Note that $D(T_1\ast\cdots\ast T_s)$ is the same as $D(T_2\ast\cdots\ast T_s\ast T_1)$, so the sub-indices in (\ref{eq:main}), (\ref{eq:3-strand}) can be cyclically permuted.

In particular, if $b_i^{\rm sw}=1$ for some $i\in\{1,2,3\}$, then with cyclical notations adopted, (\ref{eq:3-strand}) can be simplified into
\begin{align}
\Delta_{L,\rho}\doteq\frac{\det(\overline{b_{i-1}^{\rm ne}})}{\det(1-\overline{x_{i-1}^{\rm se}})}
\det\Big(\overline{b_{i+1}^{\rm ne}}+(\overline{b_i^{\rm se}}-1)\big(\overline{b_{i+1}^{\rm sw}}+\overline{b_{i-1}^{\rm sw}}\overline{b_{i-1}^{\rm ne}}^{-1}(\overline{b_{i+1}^{\rm se}}-\overline{b_{i+1}^{\rm sw}})\big)\Big).  \label{eq:3-strand-simplified}
\end{align}
This will be useful in practical computations, as done in the next section.
\end{rmk}

\section{Computations}

\begin{exmp} \label{exmp:triple-twist}
\rm Let $K$ be the triple twist knot $N([[2h_1],[2h_2],[2h_3+1]])$ given in Figure \ref{fig:triple-twist} (a).
Apply Lemma \ref{lem:algorithm} to compute $\eta_{y^{(0)}}=1$,
\begin{align*}
&\eta_{y^{(1)}}=u_{2h_1}(x,y)=[h_1]_{xy}(x-xy), \qquad \eta_{x^{(1)}}=u_{2h_1-1}(x,y)=w_1+1, \\
&\left[\begin{array}{cc} \eta_{y^{(2)}} \\ \eta_{x^{(2)}} \end{array}\right]=G_{2h_2}(y,x_1^{-1})\left[\begin{array}{cc} \eta_{y^{(0)}} \\ \eta_{x^{(1)}} \end{array}\right]=\left[\begin{array}{cc} 1+[h_2]_{yx_1^{-1}}(y-yx_1^{-1})w_1 \\ 1+(w_2+1)w_1 \end{array}\right], \\
&\left[\begin{array}{cc} \eta_{y^{(3)}} \\ \eta_{x^{(3)}} \end{array}\right]=G_{2h_3+1}(y_1,x_2)\left[\begin{array}{cc} \eta_{y^{(1)}} \\ \eta_{x^{(2)}} \end{array}\right]
=\left[\begin{array}{cc} z(w_2w_1+(xy)^{h_1})+[h_1]_{xy}(x-xy) \\ \star \\ \end{array}\right].
\end{align*}
where
$$w_1=[h_1]_{xy}(x-1), \qquad w_2=[h_2]_{yx_1^{-1}}(y-1), \qquad z=[h_3+1]_{y_1x_2}y_1-[h_3]_{y_1x_2}y_1x_2,$$
and $\star$ denotes something irrelevant to us.
By (\ref{eq:TAP-N(T)}),
\begin{align*}
&\Delta_{L,\rho}\doteq\frac{\det(\overline{\eta_{y_3}})}{\det(1-\overline{x^{\rm se}})}=\frac{\det(\overline{z}(\overline{w_2w_1}+(\overline{xy})^{h_1})+[h_1]_{\overline{xy}}(\overline{x}-\overline{xy}))}{\det(1-\overline{x})}\doteq \\
&\frac{\det\big(([h_3+1]_{\overline{y_1x_2}}\overline{y_1}-[h_3]_{\overline{y_1x_2}}\overline{y_1x_2})
([h_2]_{\overline{y}\overline{x_1}^{-1}}(\overline{y}-1)[h_1]_{\overline{x}\overline{y}}(\overline{x}-1)
+(\overline{x}\overline{y})^{h_1})+[h_1]_{\overline{x}\overline{y}}(\overline{x}-\overline{x}\overline{y})\big)}{\det(1-\overline{x})}.
\end{align*}
In particular, when $\rho=\mathbf{1}$ (the trivial representation), $\overline{x}=\overline{x_1}=\overline{x_2}=\overline{y}=\overline{y_1}=t$,
and the ordinary Alexander polynomial is easily found (referring to (\ref{eq:ordinary})):
$$\Delta_{K}\doteq ([h_3+1]_{t^2}-t[h_3]_{t^2})(h_2(t-1)^2[h_1]_{t^2}+t^{2h_1})+(1-t)[h_1]_{t^2}.$$

\begin{figure}[h]
  \centering
  \includegraphics[width=0.8\textwidth]{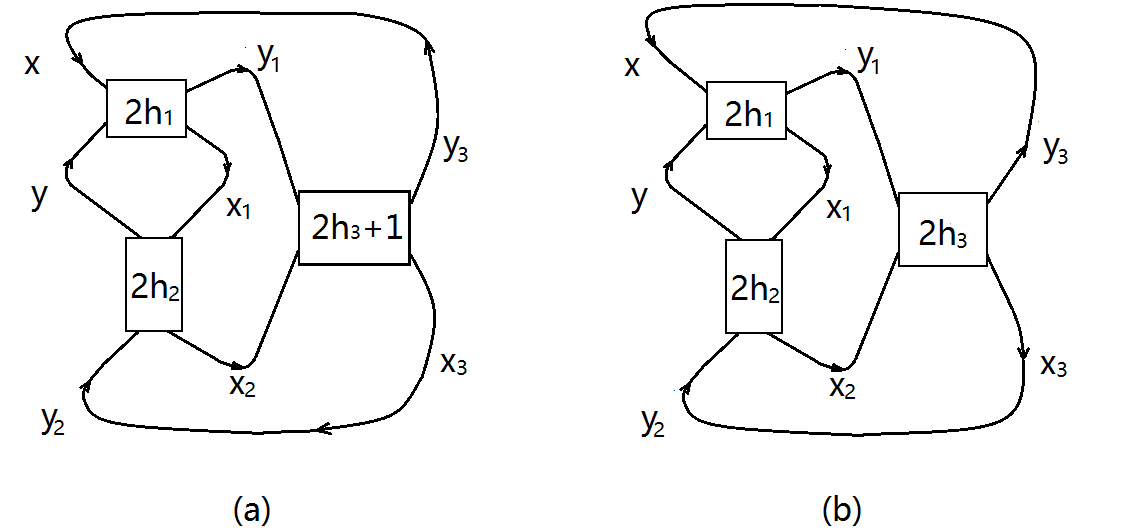}\\
  \caption{(a) The triple twist knot $N([[2h_1],[2h_2],[2h_3+1]])$; (b) The triple twist link $N([[2h_1],[2h_2],[2h_3]])$} \label{fig:triple-twist}
\end{figure}

As a special case, $N([2],[2],[3])$ (with $h_1=h_2=h_3=1$) is the knot $7_5$, so
\begin{align}
\Delta_{7_5,\rho}\doteq\frac{\det\big(\overline{y_1}(1+\overline{x_2y_1}-\overline{x_2})(1-\overline{x}-\overline{y}+\overline{y}\overline{x}+\overline{x}\overline{y})+\overline{x}(1-\overline{y})\big)}{\det(1-\overline{x})}.
\label{eq:7_5}
\end{align}
In particular, the ordinary Alexander polynomial
$$\Delta_{7_5}\doteq 2t^4-4t^3+5t^2-4t+2.$$
It is easy to see that $\pi(7_5)=\langle x,y\mid (x_2y_1)^2=yx_2y_1x_2\rangle$, with $y_1=(xy).x$, $x_2=(yx).y$.
Due to $\Delta_{7_5}(2)\equiv 0\pmod{7}$, there are homomorphisms
$$\tau_{\epsilon}:\pi(7_5)\to\mathbb{Z}_7\rtimes_2\mathbb{Z}_3=\langle\alpha,\beta\mid\alpha^7=\beta^3=1,\ \beta\alpha=\alpha^2\beta\rangle, \qquad \epsilon\in\{\pm1\},$$
sending $x$ to $\beta^{\epsilon}$ and $y$ to $\alpha\beta^{\epsilon}$.
Let $\zeta=e^{2\pi i/7}$, and let $\psi:\mathbb{Z}_7\rtimes_2\mathbb{Z}_3\to{\rm SL}(3,\mathbb{C})$ denote the representation sending $\alpha,\beta$ to $A,B$, respectively, where
$$A=\left[\begin{array}{ccc} \zeta & 0 & 0 \\ 0 & \zeta^2 & 0 \\ 0 & 0 & \zeta^4 \end{array}\right], \qquad B=\left[\begin{array}{ccc} 0 & 1 & 0 \\ 0 & 0 & 1 \\ 1 & 0 & 0 \end{array}\right].$$
Let $\rho^{7_5}_{\epsilon}:\pi(7_5)\to{\rm SL}(3,\mathbb{C})$ be the composite $\psi\circ\tau_{\epsilon}$.
We can compute $\overline{y_1}=tA^{-2\epsilon}B^{\epsilon}$, $\overline{x_2}=tA^{1+2\epsilon}B^{\epsilon}$, and then by (\ref{eq:7_5}),
\begin{align}
\Delta_{7_5,\rho^{7_5}_{\epsilon}}\doteq t^9-5t^6+5t^3-1. \label{eq:7_5-metabelian}
\end{align}
\end{exmp}

\begin{exmp}
\rm Let $L$ be the triple twist link $N([[2h_1],[2h_2],[2h_3]])$ given in Figure \ref{fig:triple-twist} (b).
Similarly as in the previous Example,
\begin{align*}
&\Delta_{L,\rho}\doteq\frac{\det([h_3]_{\overline{y_1x_2}}(\overline{y_1}-\overline{y_1x_2})(\overline{w_2w_1}+(\overline{xy})^{h_1})+[h_1]_{\overline{xy}}(\overline{x}-\overline{xy}))}{\det(1-\overline{x_3})}\doteq \\
&\frac{\det\big([h_3]_{\overline{y_1}\overline{x_2}}(\overline{y_1}-\overline{y_1}\overline{x_2})
([h_2]_{\overline{y}\overline{x_1}^{-1}}(\overline{y}-1)[h_1]_{\overline{x}\overline{y}}(\overline{x}-1)
+(\overline{x}\overline{y})^{h_1})+[h_1]_{\overline{x}\overline{y}}(\overline{x}-\overline{x}\overline{y})\big)}{\det(1-\overline{y})}.
\end{align*}
The multi-variable Alexander polynomial can be easily found by setting $\overline{x}=\overline{y_1}=t_1$ and $\overline{y}=\overline{x_1}=\overline{x_2}=t_2$:
$$\Delta_{L}\doteq [h_1]_{t_1t_2}[h_3]_{t_1t_2}(h_2(t_1-1)(t_2-1)+t_1t_2-1)+[h_1]_{t_1t_2}+[h_3]_{t_1t_2}).$$
\end{exmp}

\begin{exmp}
\rm
\begin{figure}[h]
  \centering
  \includegraphics[width=0.42\textwidth]{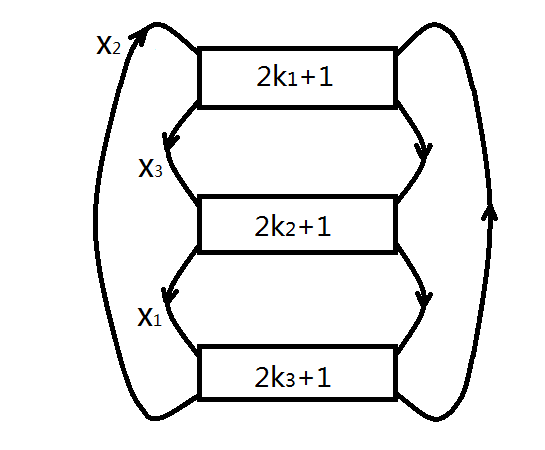}\\
  \caption{The pretzel knot $P(2k_1+1,2k_2+1,2k_3+1)$} \label{fig:pretzel}
\end{figure}

Let $K=P(2k_1+1,2k_2+1,2k_3+1)$. It is easy to see that
$$b_1^{\rm se}=u_{2k_1}(x_2,x_3^{-1}), \qquad b_2^{\rm se}=u_{2k_2}(x_3,x_1^{-1}), \qquad b_2^{\rm ne}=u_{2k_2+1}(x_3,x_1^{-1}),$$
and $b_i^{\rm sw}=1$, $i=1,2,3$.
Let
$$z_i=[k_i]_{x_{i+1}x_{i-1}^{-1}}x_{i+1}(1-x_{i-1}^{-1})-1, \qquad
w_i=[k_i+1]_{x_{i+1}x_{i-1}^{-1}}(x_{i+1}-1)+1.$$
Then by (\ref{eq:3-strand-simplified}),
\begin{align}
\Delta_{K,\rho}\doteq\frac{\det(\overline{w_3})}{\det(1-\overline{x_1})}\det\big(\overline{w_2}+\overline{z_1}(1+\overline{w_3}^{-1}\overline{z_2})\big).   \label{eq:pretzel}
\end{align}
In particular, the ordinary Alexander polynomial is
$$\Delta_K\doteq1+(1+k_1+k_2+k_3+k_1k_2+k_2k_3+k_1k_3)(t+t^{-1}-2).$$

Consider the case $k_1=k_2=k_3=1$. We have
\begin{align}
\overline{z_1}&=t\rho(x_2)-\rho(x_2x_3^{-1})-1, \label{eq:pretzel-z1} \\
\overline{z_2}&=t\rho(x_3)+t\rho(x_3x_1^{-1})-1, \label{eq:pretzel-z2} \\
\overline{w_2}&=t\rho(x_3)+t\rho(x_3x_1^{-1}x_3)+\rho(x_3x_1^{-1}), \label{eq:pretzel-w2} \\
\overline{w_3}&=t\rho(x_1)+t\rho(x_1x_2^{-1}x_1)-\rho(x_1x_2^{-1}).  \label{eq:pretzel-w3}
\end{align}
The results of \cite{Ch18} enable us to easily construct irreducible 2-dimensional representations of $\pi(K)$. For instance, the assignment
$$x_1\mapsto \left[\begin{array}{cc} 1 & 1 \\ 0 & 1 \end{array}\right], \qquad
x_2\mapsto \left[\begin{array}{cc} 1 & 0 \\ -3 & 1 \end{array}\right], \qquad
x_3\mapsto\left[\begin{array}{cc} 2 & 1 \\ -1 & 0 \end{array}\right]$$
defines a representation $\rho:\pi(K)\to{\rm SL}(2,\mathbb{F}_{11})$, as can be checked using (16)--(18) of \cite{Ch18}.
By (\ref{eq:pretzel}) and (\ref{eq:pretzel-z1})--(\ref{eq:pretzel-w3}), we can compute
$$\Delta_{P(3,3,3),\rho}\doteq\frac{(t+2)(t+4)(t^2+3t-2)}{(t-1)^2}\in\mathbb{F}_{11}(t).$$
\end{exmp}

\begin{exmp}
\rm Let $L$ be the 3-component pretzel link $P(2k_1,2k_2,2k_3)$.

\begin{figure}[h]
  \centering
  \includegraphics[width=0.4\textwidth]{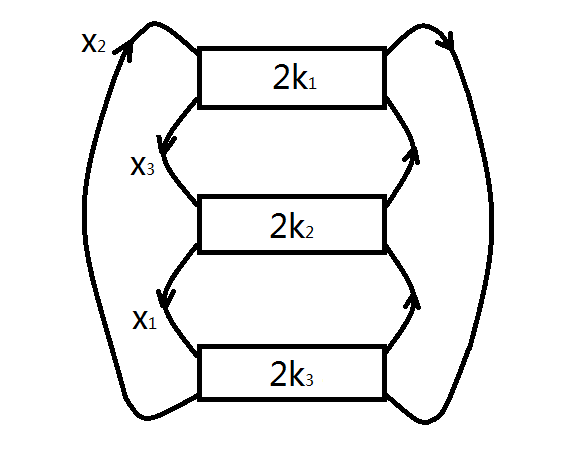}\\
  \caption{The pretzel link $P(2k_1,2k_2,2k_3)$}\label{fig:pretzel-link}
\end{figure}

For $i=1,2,3$, let $y_i=x_{i+1}x_{i-1}^{-1}$, and $v_i=[k_i]_{y_i}$, then $b_i^{\rm sw}=1$, $b_i^{\rm ne}=v_i(x_{i+1}-y_{i})$, $b_i^{\rm se}=v_i(x_{i+1}-1)+1$.
By (\ref{eq:3-strand-simplified}),
$$\Delta_{L,\rho}\doteq\frac{\det(\overline{v_3}(\overline{x_1}-\overline{y_3}))}{\det(1-\overline{x_2})}
\det\Big(\overline{v_2}(\overline{x_3}-\overline{y_2})+\overline{v_1}(\overline{x_2}-1)+\overline{v_1}(\overline{v_3}\overline{y_3})^{-1}\overline{v_2}(\overline{x_3}-1)\Big).$$
Specialized to $\rho=\mathbf{1}$ so that $\overline{x_i}=t_i$, the multi-variable Alexander polynomial is easily found:
$$\Delta_{L}\doteq \sum\limits_{i=1}^3t_{i-1}(t_i-1)[k_{i+1}]_{t_{i-1}t_i^{-1}}[k_{i-1}]_{t_it_{i+1}^{-1}}.$$

As a special case when $k_1=k_2=k_3=1$, $L=P(2,2,2)$ is the 3-chain link, a very interesting link whose complement in $S^3$ is named the ``magic manifold".
The formula can be simplified into
$$\Delta_{P(2,2,2),\rho}\doteq\det((\overline{x_2}^{-1}+\overline{x_1}^{-1})(\overline{x_3}-1)-\overline{x_2}^{-1}\overline{x_3}\overline{x_1}^{-1}+1).$$
\end{exmp}

\begin{exmp}
\rm Let $K=M(7/3,2,7/2)$, which is the mirror of $12a\_0423$ in the knot table. Write $K=D(T_1\ast T_2\ast T_3)$, with $T_1=[[3],[2]]$, $T_2=[2]$, $T_3=[[2],[3]]$.
Use $x,y,z$ to express other elements:
\begin{align*}
x'=z^{-1}.x, \qquad y'=(xy^{-1}x).y, \qquad z'=(xz)^{-1}.z, \qquad w=(xy^{-1}).x.
\end{align*}

\begin{figure}[h]
  \centering
  \includegraphics[width=0.42\textwidth]{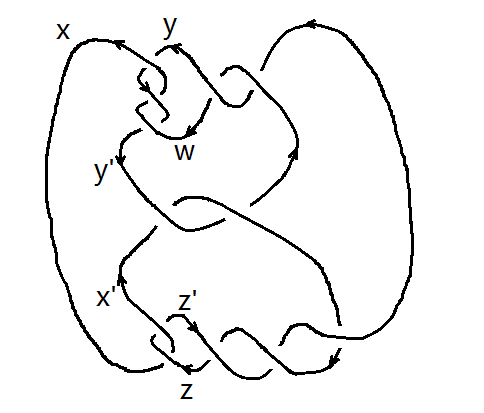}\\
  \caption{The knot $M(7/3,2,7/2)$}\label{fig:12a-0423}
\end{figure}

Computing routinely,
\begin{align*}
b_1^{\rm sw}&=u_3(x,y^{-1})=x(1+y^{-1}x-y^{-1}), \\
\left[\begin{array}{cc} b_1^{\rm ne} \\ b_1^{\rm se} \end{array}\right]&=G_2(y^{-1},w^{-1})\left[\begin{array}{cc} 1 \\ u_2(x,y^{-1}) \end{array}\right]
=\left[\begin{array}{cc} 1+y^{-1}(1-w^{-1})(x-xy^{-1}-1) \\ \star \end{array}\right], \\
b_2^{\rm se}&=y', \\
b_3^{\rm sw}&=u_2(x',z')=x'(1-z'), \\
\left[\begin{array}{cc} b_3^{\rm ne} \\ b_3^{\rm se} \end{array}\right]&=G_3(z',z^{-1})\left[\begin{array}{cc} 1 \\ x' \end{array}\right]
=\left[\begin{array}{cc} 1+z'(1+z^{-1}z'-z^{-1})(x'-1) \\ 1+z'(1-z^{-1})(x'-1) \end{array}\right].
\end{align*}

Then $\Delta_{K,\rho}$ can be obtained via (\ref{eq:3-strand-simplified}) by plugging in these.

Specialized to $\rho=\mathbf{1}$,
$$\overline{b_1^{\rm sw}}=2t-1, \quad \overline{b_2^{\rm se}}=t, \quad \overline{b_3^{\rm sw}}=t-t^2, \qquad t^2\overline{b_1^{\rm ne}}=\overline{b_3^{\rm ne}}=\overline{b_3^{\rm se}}-\overline{b_3^{\rm sw}}=2t^2-3t+2,$$
leading us to
$$\Delta_K\doteq (2t^2-3t+2)(2t^4-4t^3+5t^2-4t+2)\doteq (2t^2-3t+2)\Delta_{7_5}.$$ %=4t^6-14t^5+26t^4-31t^3+26t^2-14t+4.$$

Similarly as in Example \ref{exmp:triple-twist}, we can consider representations $\pi(K)\to\mathbb{Z}_7\rtimes_2\mathbb{Z}_3$.
The relations of $\pi(K)$ are
$$(wy)^{-1}.y=(z'z^{-1}).z', \qquad y^{-1}.w=(y'x').y',$$
from which it follows that the assignment $x\mapsto\beta$, $y\mapsto\alpha^a\beta$, $z\mapsto\alpha^b\beta$
defines a nonabelian representation $\pi(K)\to\mathbb{Z}_7\rtimes_2\mathbb{Z}_3$ if and only if $b\equiv a\not\equiv0\pmod{7}$.
Let $\rho^K_a:\pi(K)\to{\rm SL}(3,\mathbb{C})$ denote the composite of this representation with $\psi$ (given in Example \ref{exmp:triple-twist}).
With the help of Mathematica, we compute
\begin{align*}
\Delta_{K,\rho^K_a}\doteq&\ \frac{1}{t^3-1}\Big(t^{18}-(2\zeta^{5a}+3\zeta^{4a}+\zeta^{3a}+4\zeta^{2a}+3\zeta^a+5)t^{15}+(11\zeta^{5a}+2\zeta^{4a}  \\
&+10\zeta^{3a}+10\zeta^{2a}+7\zeta^a+8)t^{12}-(4\zeta^{5a}+3\zeta^{4a}+8\zeta^{3a}+6\zeta^{2a}-\zeta^a)t^9-   \\
&(4\zeta^{5a}+\zeta^{3a}+5\zeta^{2a}+3\zeta^a+9)t^6+(\zeta^{5a}+2\zeta^{3a}+3\zeta^{2a}+2\zeta^a+2)t^3+1\Big).
\end{align*}

Assume there exists an epimorphism $\phi:\pi(K)\twoheadrightarrow\pi(7_5)$. Then $\rho^{7_5}_{\varepsilon}\circ\phi$ is conjugate to $\rho^K_a$ for some $a$ and $\varepsilon\in\{\pm1\}$. However, looking back to (\ref{eq:7_5-metabelian}), we see that $\Delta_{7_5,\rho^{7_5}_{\pm1}}\nmid\Delta_{K,\rho^K_a}$. Thus, while the ordinary Alexander polynomial does not detect $K\ngeqslant 7_5$, TAP does.
\end{exmp}

\end{document}